\documentclass{amsart}
\usepackage{amsfonts}
\usepackage{amsmath,amscd,lscape}
\usepackage{amsthm}
\usepackage{amssymb}
\usepackage{latexsym}

\usepackage[curve,matrix,arrow,color]{xy}

\usepackage[usenames,dvipsnames]{color}

\newtheorem{cor}{Corollary}[section]
\newtheorem{lem}{Lemma}[section]

\newtheorem{prop}{Proposition}[section]

\newtheorem{example}{Example}
\newtheorem{defn}{Definition}[section]
\newtheorem{thm}{Theorem}
\newtheorem{conj}{Conjecture}
\newtheorem{rem}{Remark}
\newtheorem{notation}{Notation}

\numberwithin{equation}{section}
\begin{document}
\newcommand{\beqa}{\begin{eqnarray}}
\newcommand{\eeqa}{\end{eqnarray}}
\newcommand{\thmref}[1]{Theorem~\ref{#1}}
\newcommand{\secref}[1]{Sect.~\ref{#1}}
\newcommand{\lemref}[1]{Lemma~\ref{#1}}
\newcommand{\propref}[1]{Proposition~\ref{#1}}
\newcommand{\corref}[1]{Corollary~\ref{#1}}
\newcommand{\remref}[1]{Remark~\ref{#1}}
\newcommand{\er}[1]{(\ref{#1})}
\newcommand{\nc}{\newcommand}
\newcommand{\rnc}{\renewcommand}

\nc{\cal}{\mathcal}

\nc{\goth}{\mathfrak}
\rnc{\bold}{\mathbf}
\renewcommand{\frak}{\mathfrak}
\renewcommand{\Bbb}{\mathbb}

\def\cN{{\cal N}}
\newcommand{\hs}[1]{\hspace{#1 mm}}
\newcommand{\mb}[1]{\hs{4}\mbox{#1}\hs{4}}
\newcommand{\id}{\text{id}}
\nc{\Cal}{\mathcal}
\nc{\Xp}[1]{X^+(#1)}
\nc{\Xm}[1]{X^-(#1)}
\nc{\on}{\operatorname}
\nc{\ch}{\mbox{ch}}
\nc{\Z}{{\bold Z}}
\nc{\J}{{\mathcal J}}
\nc{\C}{{\bold C}}
\nc{\Q}{{\bold Q}}
\renewcommand{\P}{{\mathcal P}}
\nc{\N}{{\Bbb N}}
\nc\beq{\begin{equation}}
\nc\enq{\end{equation}}
\nc\lan{\langle}
\nc\ran{\rangle}
\nc\bsl{\backslash}
\nc\mto{\mapsto}
\nc\lra{\leftrightarrow}
\nc\hra{\hookrightarrow}
\nc\sm{\smallmatrix}
\nc\esm{\endsmallmatrix}
\nc\sub{\subset}
\nc\ti{\tilde}
\nc\nl{\newline}
\nc\fra{\frac}
\nc\und{\underline}
\nc\ov{\overline}
\nc\ot{\otimes}
\nc\bbq{\bar{\bq}_l}
\nc\bcc{\thickfracwithdelims[]\thickness0}
\nc\ad{\text{\rm ad}}
\nc\Ad{\text{\rm Ad}}
\nc\Hom{\text{\rm Hom}}
\nc\End{\text{\rm End}}
\nc\Ind{\text{\rm Ind}}
\nc\Res{\text{\rm Res}}
\nc\Ker{\text{\rm Ker}}
\rnc\Im{\text{Im}}
\nc\sgn{\text{\rm sgn}}
\nc\tr{\text{\rm tr}}
\nc\Tr{\text{\rm Tr}}
\nc\supp{\text{\rm supp}}
\nc\card{\text{\rm card}}
\nc\bst{{}^\bigstar\!}
\nc\he{\heartsuit}
\nc\clu{\clubsuit}
\nc\spa{\spadesuit}
\nc\di{\diamond}
\nc\cW{\cal W}
\nc\cG{\cal G}
\nc\al{\alpha}
\nc\bet{\beta}
\nc\ga{\gamma}
\nc\de{\delta}
\nc\ep{\epsilon}
\nc\io{\iota}
\nc\om{\omega}
\nc\si{\sigma}
\rnc\th{\theta}
\nc\ka{\kappa}
\nc\la{\lambda}
\nc\ze{\zeta}

\nc\vp{\varpi}
\nc\vt{\vartheta}
\nc\vr{\varrho}

\nc\Ga{\Gamma}
\nc\De{\Delta}
\nc\Om{\Omega}
\nc\Si{\Sigma}
\nc\Th{\Theta}
\nc\La{\Lambda}

\nc\boa{\bold a}
\nc\bob{\bold b}
\nc\boc{\bold c}
\nc\bod{\bold d}
\nc\boe{\bold e}
\nc\bof{\bold f}
\nc\bog{\bold g}
\nc\boh{\bold h}
\nc\boi{\bold i}
\nc\boj{\bold j}
\nc\bok{\bold k}
\nc\bol{\bold l}
\nc\bom{\bold m}
\nc\bon{\bold n}
\nc\boo{\bold o}
\nc\bop{\bold p}
\nc\boq{\bold q}
\nc\bor{\bold r}
\nc\bos{\bold s}
\nc\bou{\bold u}
\nc\bov{\bold v}
\nc\bow{\bold w}
\nc\boz{\bold z}

\nc\ba{\bold A}
\nc\bb{\bold B}
\nc\bc{\bold C}
\nc\bd{\bold D}
\nc\be{\bold E}
\nc\bg{\bold G}
\nc\bh{\bold H}
\nc\bi{\bold I}
\nc\bj{\bold J}
\nc\bk{\bold K}
\nc\bl{\bold L}
\nc\bm{\bold M}
\nc\bn{\bold N}
\nc\bo{\bold O}
\nc\bp{\bold P}
\nc\bq{\bold Q}
\nc\br{\bold R}
\nc\bs{\bold S}
\nc\bt{\bold T}
\nc\bu{\bold U}
\nc\bv{\bold V}
\nc\bw{\bold W}
\nc\bz{\bold Z}
\nc\bx{\bold X}

\nc\ca{\mathcal A}
\nc\cb{\mathcal B}
\nc\cc{\mathcal C}
\nc\cd{\mathcal D}
\nc\ce{\mathcal E}
\nc\cf{\mathcal F}
\nc\cg{\mathcal G}
\rnc\ch{\mathcal H}
\nc\ci{\mathcal I}
\nc\cj{\mathcal J}
\nc\ck{\mathcal K}
\nc\cl{\mathcal L}
\nc\cm{\mathcal M}
\nc\cn{\mathcal N}
\nc\co{\mathcal O}
\nc\cp{\mathcal P}
\nc\cq{\mathcal Q}
\nc\car{\mathcal R}
\nc\cs{\mathcal S}
\nc\ct{\mathcal T}
\nc\cu{\mathcal U}
\nc\cv{\mathcal V}
\nc\cz{\mathcal Z}
\nc\cx{\mathcal X}
\nc\cy{\mathcal Y}

\nc\e[1]{E_{#1}}
\nc\ei[1]{E_{\delta - \alpha_{#1}}}
\nc\esi[1]{E_{s \delta - \alpha_{#1}}}
\nc\eri[1]{E_{r \delta - \alpha_{#1}}}
\nc\ed[2][]{E_{#1 \delta,#2}}
\nc\ekd[1]{E_{k \delta,#1}}
\nc\emd[1]{E_{m \delta,#1}}
\nc\erd[1]{E_{r \delta,#1}}

\nc\ef[1]{F_{#1}}
\nc\efi[1]{F_{\delta - \alpha_{#1}}}
\nc\efsi[1]{F_{s \delta - \alpha_{#1}}}
\nc\efri[1]{F_{r \delta - \alpha_{#1}}}
\nc\efd[2][]{F_{#1 \delta,#2}}
\nc\efkd[1]{F_{k \delta,#1}}
\nc\efmd[1]{F_{m \delta,#1}}
\nc\efrd[1]{F_{r \delta,#1}}

\nc\fa{\frak a}
\nc\fb{\frak b}
\nc\fc{\frak c}
\nc\fd{\frak d}
\nc\fe{\frak e}
\nc\ff{\frak f}
\nc\fg{\frak g}
\nc\fh{\frak h}
\nc\fj{\frak j}
\nc\fk{\frak k}
\nc\fl{\frak l}
\nc\fm{\frak m}
\nc\fn{\frak n}
\nc\fo{\frak o}
\nc\fp{\frak p}
\nc\fq{\frak q}
\nc\fr{\frak r}
\nc\fs{\frak s}
\nc\ft{\frak t}
\nc\fu{\frak u}
\nc\fv{\frak v}
\nc\fz{\frak z}
\nc\fx{\frak x}
\nc\fy{\frak y}

\nc\fA{\frak A}
\nc\fB{\frak B}
\nc\fC{\frak C}
\nc\fD{\frak D}
\nc\fE{\frak E}
\nc\fF{\frak F}
\nc\fG{\frak G}
\nc\fH{\frak H}
\nc\fJ{\frak J}
\nc\fK{\frak K}
\nc\fL{\frak L}
\nc\fM{\frak M}
\nc\fN{\frak N}
\nc\fO{\frak O}
\nc\fP{\frak P}
\nc\fQ{\frak Q}
\nc\fR{\frak R}
\nc\fS{\frak S}
\nc\fT{\frak T}
\nc\fU{\frak U}
\nc\fV{\frak V}
\nc\fZ{\frak Z}
\nc\fX{\frak X}
\nc\fY{\frak Y}
\nc\tfi{\ti{\Phi}}
\nc\bF{\bold F}
\rnc\bol{\bold 1}

\nc\ua{\bold U_\A}

%%%%%%%%%%%%%%%%%%%%%%%
%%%%%    CALIGRAPHIQUES   %%%%
%%%%%%%%%%%%%%%%%%%%%%
\def\cA{{\cal A}}   \def\cB{{\cal B}}   \def\cC{{\cal C}}
\def\cD{{\cal D}}   \def\cE{{\cal E}}   \def\cF{{\cal F}}  
\def\cG{{\cal G}}   \def\cH{{\cal H}}   \def\cI{{\cal I}}
\def\cJ{{\cal J}}   \def\cK{{\cal K}}   \def\cL{{\cal L}}
\def\cM{{\cal M}}   \def\cN{{\cal N}}   \def\cO{{\cal O}}
\def\cP{{\cal P}}   \def\cQ{{\cal Q}}   \def\cR{{\cal R}}
\def\cS{{\cal S}}   \def\cT{{\cal T}}   \def\cU{{\cal U}}
\def\cV{{\cal V}}   \def\cW{{\cal W}}   \def\cX{{\cal X}}
\def\cY{{\cal Y}}   \def\cZ{{\cal Z}}  \def\cRR{{\cal {\mathbb R}}}

\newcommand{\ben}{\begin{eqnarray}}
\newcommand{\een}{\end{eqnarray}}
\newcommand{\nonu}{\nonumber \\} 
\newcommand{\FF}{{\mathbb F}}
\newcommand{\A}{{\mathbb A}}
\newcommand{\GG}{{\mathbb G}}

%%%%%%%%%%%%%%%%%%%%%%%%%%%%%%%%%%%%%%%%%%%%%%%%%%%%%%
\nc\qinti[1]{[#1]_i}
\nc\q[1]{[#1]_q}
\nc\xpm[2]{E_{#2 \delta \pm \alpha_#1}}  %\xpm{j}{l}
\nc\xmp[2]{E_{#2 \delta \mp \alpha_#1}}
\nc\xp[2]{E_{#2 \delta + \alpha_{#1}}}
\nc\xm[2]{E_{#2 \delta - \alpha_{#1}}}
\nc\hik{\ed{k}{i}}
\nc\hjl{\ed{l}{j}}
\nc\qcoeff[3]{\left[ \begin{smallmatrix} {#1}& \\ {#2}& \end{smallmatrix}
\negthickspace \right]_{#3}}
\nc\qi{q}
\nc\qj{q}

\nc\ufdm{{_\ca\bu}_{\rm fd}^{\le 0}}

%%%%%%%%%%%%%%%%%%%%%%%%%%%%%%%%%%%%%%%%%%%%%%%%%%%%%%

%\nc\rtimes
\nc\isom{\cong} 

\nc{\pone}{{\Bbb C}{\Bbb P}^1}
\nc{\pa}{\partial}
\def\H{\mathcal H}
\def\L{\mathcal L}
\nc{\F}{{\mathcal F}}
\nc{\Sym}{{\goth S}}
%\nc{\A}{{\mathcal A}}
\nc{\arr}{\rightarrow}
\nc{\larr}{\longrightarrow}

\nc{\ri}{\rangle}
\nc{\lef}{\langle}
\nc{\W}{{\mathcal W}}
\nc{\uqatwoatone}{{U_{q,1}}(\su)}
\nc{\uqtwo}{U_q(\goth{sl}_2)}
\nc{\dij}{\delta_{ij}}
\nc{\divei}{E_{\alpha_i}^{(n)}}
\nc{\divfi}{F_{\alpha_i}^{(n)}}
\nc{\Lzero}{\Lambda_0}
\nc{\Lone}{\Lambda_1}
\nc{\ve}{\varepsilon}
\nc{\phioneminusi}{\Phi^{(1-i,i)}}
\nc{\phioneminusistar}{\Phi^{* (1-i,i)}}
\nc{\phii}{\Phi^{(i,1-i)}}
\nc{\Li}{\Lambda_i}
\nc{\Loneminusi}{\Lambda_{1-i}}
\nc{\vtimesz}{v_\ve \otimes z^m}

\nc{\asltwo}{\widehat{\goth{sl}_2}}
\nc\ag{\widehat{\goth{g}}}  
\nc\teb{\tilde E_\boc}
\nc\tebp{\tilde E_{\boc'}}

\title[An attractive basis for the $q-$Onsager algebra]{An attractive basis for the $q-$Onsager algebra}
%\dedicatory{}
\author{Pascal Baseilhac}
\address{Laboratoire de Math\'ematiques et Physique Th\'eorique CNRS/UMR 6083,
          F\'ed\'eration Denis Poisson, Universit\'e de Tours, Parc de Grammont, 37200 Tours, FRANCE}
\email{baseilha@lmpt.univ-tours.fr}

\author{Samuel Belliard}
\address{Institut de Physique Th\'eorique, DSM, CEA, URA2306 CNRS Saclay, F-91191 Gif-sur-Yvette, FRANCE}
\email{samuel.belliard@cea.fr}

\begin{abstract}  Let $\textsf{A},\textsf{A}^*$ be the fundamental generators of the $q-$Onsager algebra. A linear basis for the $q-$Onsager algebra is known as the `zig-zag' basis \cite{IT10}. In this letter, an attractive basis for the $q-$Onsager algebra is conjectured, based on the relation between the $q-$Onsager algebra and a quotient of the infinite dimensional algebra ${\cal A}_q$ introduced in \cite{BK}.
\end{abstract}

\maketitle

\vskip -0.5cm


\vspace{0cm}

\vspace{0mm}

\section{Introduction}
Introduced in the mathematical physics literature, the Onsager algebra (OA) admits two presentations. A first presentation is given in terms of two generators $A_0,A_1$ subject to a pair of relations, the so-called Dolan-Grady relations \cite{DG}:
\beqa
[A_0,[A_0,[A_0,A_1]]]=16[A_0,A_1], \qquad [A_1,[A_1,[A_1,A_0]]]=16[A_1,A_0].\label{OAP1}
\eeqa
A second presentation originates in Onsager's work on the exact solution of the two-dimensional Ising model \cite{Ons44}. It is given in terms of generators $\{A_k,G_l|k,l\in{\mathbb Z}\}$ and relations:
\beqa
\big[A_{k},A_{l}\big]= 4G_{k-l},\quad \big[G_l,A_k\big] =2A_{k+l}-2A_{k-l},\quad \big[G_k,G_l\big] =0.\label{OAP2}
\eeqa
In the 90's, the isomorphism between these first and second presentations was established \cite{Davies,Roan}, and generators $\{A_k,G_l|k,l\in{\mathbb Z}\}$ were systematically written as polynomials in $A_0,A_1$. For a review of the proof, see \cite{EC12}. Note that the Onsager algebra has also a presentation as an invariant subalgebra of the loop algebra ${\mathbb C}[t,t^{-1}]\otimes sl_2$ by an involution \cite{Davies,Roan}. \vspace{1mm} 

The $q-$Onsager algebra ($q-$OA) has been introduced as a $q-$deformed analog of the presentation (\ref{OAP1}). In this presentation, the defining relations are given in Definition \ref{qons}. It first appeared in the mathematical literature  in the context of $P-$ and $Q-$polynomial schemes, as a special case of the so-called tridiagonal algebra \cite{Ter03}. It is now well-understood that the
representation theory of tridiagonal algebras - in particular the
$q-$OA - is intimately connected with the theory of Leonard pairs and tridiagonal pairs  developped by Terwilliger and coauthors (see all the references citing \cite{Ter03}). Moreover, the deformation scheme of the Onsager algebra to the $q-$Onsager algebra is given by Manin triples technique \cite{BC12}.
Independently, the $q-$OA also appeared in the context of quantum integrable systems \cite{B1,BK}, the spectral parameter dependent reflection equation algebra \cite{B1,BSh1} and as a certain coideal subalgebra of $U_q(\widehat{sl_2})$ \cite{BB}.  \vspace{1mm} 

By analogy with the OA, besides the first presentation associated with (\ref{qDG}) it was expected that the $q-$OA  admits a second presentation, that would be associated with an infinite number of generators satisfying certain $q-$deformed analogs of (\ref{OAP2}). In this direction, one strategy\footnote{Another strategy is based on the braid group action of the $q-$Onsager algebra \cite{BK4}.} was to study in details the structure of more general (non-scalar) solutions of the reflection equation algebra, the so-called Sklyanin operators \cite{Skly88}. An 
 infinite dimensional algebra called  ${\cal A}_q$ with generators $\{{\cal W}_{-k},{\cal W}_{k+1},{\cal G}_{k+1},\tilde{\cal G}_{k+1}|k\in{\mathbb Z}_+\}$ and relations (see Definition \ref{defnCA}) was introduced in \cite{BSh1}, based on the results of \cite{BK}. In particular, for the explicit examples of Sklyanin operators considered in \cite{BK} it was observed that the first generators ${\cal W}_0,{\cal W}_1$  satisfy the pair of relations (\ref{qDG}). By a brute force calculation, it was also observed that the eight generators denoted ${\cal W}_{-1},{\cal W}_{-2},{\cal W}_{2},{\cal W}_{3},{\cal G}_{1},{\cal G}_{2},\tilde{\cal G}_{1},\tilde{\cal G}_{2}$ can be written as polynomials in ${\cal W}_0,{\cal W}_1$ only \cite[see eqs. (49)]{BK}.  This strongly suggested that the algebra  ${\cal A}_q$ or certain of its quotients are natural candidates for a second presentation of the $q-$OA.\vspace{1mm}
 
 In this letter, we investigate the relationship between the $q-$Onsager algebra and a class of quotients of the infinite dimensional algebra ${\cal A}_q$, denoted $\tilde{\cal A}_q^{\{\delta\}}$ (see Definition \ref{tildeAq}). It is shown that all generators of $\tilde{\cal A}_q^{\{\delta\}}$ are polynomials in ${\cal W}_0,{\cal W}_1$ (see Corollary \ref{cor31}), explicitly determined through recursive formulae (see Proposition \ref{prop1} and Example \ref{ex1}). Then, we propose a Poincar\'e-Birkhoff-Witt (PBW) basis for the algebra ${\cal A}_q^{\{\delta\}}$ (see Conjecture \ref{conj1}) in Section 4. Based on the existence of an homomorphism from $\tilde{\cal A}_q^{\{\delta\}}$ to the $q-$Onsager algebra (see Proposition \ref{propmap}) as well as a detailed comparison between the PBW basis for ${\cal A}_q^{\{\delta\}}$ and the so-called `zig-zag' basis introduced by Ito and Terwilliger in \cite{IT10}, it is conjectured that the algebra ${\cal A}_q^{\{\delta\}}$ and the $q-$Onsager algebra are isomorphic, see Conjecture \ref{conj2}. In other words, we claim that the algebra $\tilde{\cal A}_q^{\{\delta\}}$ gives a second presentation for the $q-$Onsager algebra.
 \vspace{1mm}
\begin{notation} We fix a field ${\mathbb K}$ and a nonzero $q\in {\mathbb K}$ assumed not to be a root of unity. We denote ${\mathbb Z}_+$ the set of nonnegative integers. We denote $[X,Y]_q=qXY-q^{-1}YX$. In the text, $u$ and $v$ denote indeterminates. Let $n$ be a positive integer. We use the notation $[n]_q=(q^n-q^{-n})/(q-q^{-1})$.
\end{notation}

\section{The algebra ${\cal A}_q$ and central elements}
In this section, the infinite dimensional quantum algebra ${\cal A}_q$ introduced in \cite{BSh1} is recalled through generators and relations.  A generating function for central elements of ${\cal A}_q$  is constructed, that will play a central role in the analysis of further sections. 
\begin{defn}[See \cite{BSh1}]{\label{defnCA}}  ${\cal A}_q$ is an associative algebra with unit $1$, generators $\{{\cW}_{-k},{\cW}_{k+1},{\cG}_{k+1},{\tilde{\cG}}_{k+1}|k\in{\mathbb Z}_+\}$ and nonzero scalar $\rho\in{\mathbb K}$. Define $U=(qu^2+q^{-1}u^{-2})/(q+q^{-1})$ and $V=(qv^2+q^{-1}v^{-2})/(q+q^{-1})$.
Introduce the formal power series:
\begin{align}
{\cW}_+(u)=\sum_{k\in {\mathbb Z}_+}{\cW}_{-k}U^{-k-1} \ , \quad {\cW}_-(u)=\sum_{k\in {\mathbb Z}_+}{\cW}_{k+1}U^{-k-1} \ ,\label{c1}\\
 \quad {\cG}_+(u)=\sum_{k\in {\mathbb Z}_+}{\cG}_{k+1}U^{-k-1} \ , \quad {\cG}_-(u)=\sum_{k\in {\mathbb Z}_+}\tilde{{\cG}}_{k+1}U^{-k-1} \ .\label{c2}
\end{align}
The defining relations are:
\begin{align}
&&\big[{\cW}_\pm(u),{\cW}_\pm(v)\big]=0\ ,\qquad\qquad\qquad\qquad\qquad\qquad\qquad\label{ec1}\\
&&\big[{\cW}_+(u),{\cW}_-(v)\big]+\big[{\cW}_-(u),{\cW}_+(v)\big]=0\ ,\qquad\qquad\qquad\qquad\qquad\qquad\qquad\label{ec3}\\
%\end{align}
%
%\begin{align}
&&(U-V)\big[{\cW}_\pm(u),{\cW}_\mp(v)\big]= \frac{(q-q^{-1})}{\rho(q+q^{-1})}\left({\cG}_\pm(u){\cG}_\mp(v)-{\cG}_\pm(v){\cG}_\mp(u)\right)\qquad\qquad\qquad\label{ec4}\\
&& \qquad \qquad\qquad\qquad\qquad\qquad\qquad\qquad+ \frac{1}{(q+q^{-1})} \big({\cG}_\pm(u)-{\cG}_\mp(u)+{\cG}_\mp(v)-{\cG}_\pm(v)\big)\ ,\nonumber\\
&&{\cW}_\pm(u){\cW}_\pm(v)-{\cW}_\mp(u){\cW}_\mp(v)+\frac{1}{\rho(q^2-q^{-2})}\big[{\cG}_\pm(u),{\cG}_\mp(v)\big]\qquad\qquad\qquad\qquad \qquad\label{ec5}\\
&&\qquad\qquad\qquad\qquad\qquad\qquad+ \ \frac{1-UV}{U-V}\big({\cW}_\pm(u){\cW}_\mp(v)-{\cW}_\pm(v){\cW}_\mp(u)\big)=0\ ,\nonumber
\end{align}
\begin{align}
&&U\big[{\cG}_\mp(v),{\cW}_\pm(u)\big]_q -V\big[{\cG}_\mp(u),{\cW}_\pm(v)\big]_q - (q-q^{-1})\big({\cW}_\mp(u){\cG}_\mp(v)-{\cW}_\mp(v){\cG}_\mp(u)\big)\label{ec6}\\
&&\qquad\qquad\qquad\qquad\qquad\qquad\qquad \quad + \ \rho \big(U{\cW}_\pm(u)-V{\cW}_\pm(v)-{\cW}_\mp(u)+{\cW}_\mp(v)\big)=0\ ,\nonumber\\
&&U\big[{\cW}_\mp(u),{\cG}_\mp(v)\big]_q -V\big[{\cW}_\mp(v),{\cG}_\mp(u)\big]_q - (q-q^{-1})\big({\cW}_\pm(u){\cG}_\mp(v)-{\cW}_\pm(v){\cG}_\mp(u)\big)\label{ec7}\\
&& \qquad\qquad\qquad\qquad\qquad\qquad\qquad\quad+  \ \rho \big(U{\cW}_\mp(u)-V{\cW}_\mp(v)-{\cW}_\pm(u)+{\cW}_\pm(v)\big)=0\ ,\nonumber\\
&&\big[{\cG}_\epsilon(u),{\cW}_\pm(v)\big]+\big[{\cW}_\pm(u),{\cG}_\epsilon(v)\big]=0 \ ,\quad \forall \epsilon=\pm\label{ec8}\ ,\qquad\qquad\qquad\qquad\qquad\qquad\qquad\\
&&\big[{\cG}_\pm(u),{\cG}_\pm(v)\big]=0\ ,\label{ec9}\qquad\qquad\qquad\qquad\qquad\qquad\qquad\\ 
&&\big[{\cG}_+(u),{\cG}_-(v)\big]+\big[{\cG}_-(u),{\cG}_+(v)\big]=0\ .\qquad\qquad\qquad\qquad\qquad\qquad\qquad\ \label{ec16}
\end{align}
%\eeqa
\end{defn}
\begin{rem}\label{remomega}
There exists an automorphism $\Omega$ in ${\cal A}_q$:
\beqa
\Omega({\cW}_{-k})={\cW}_{k+1}\ , \quad \Omega({\cW}_{k+1})={\cW}_{-k}\ , \quad \Omega({\cG}_{k+1})=\tilde{\cG}_{k+1}\ , \quad \Omega(\tilde{\cG}_{k+1})={\cG}_{k+1} \ .\nonumber  
\eeqa
\end{rem}

%\begin{rem}\label{remtau} There exists an automorphism $\tau$ defined by:
%
%\beqa
%\tau(\cW_\pm(u))=\cW_\pm(-q^{-1}u^{-1}) \ , \qquad  \tau(\cG_\pm(u))=\cG_\pm(-q^{-1}u^{-1}). 
%\eeqa
%\end{rem}
%
\vspace{1mm}

 In ${\cal A}_q$, an infinite number of central elements can be constructed using the explicit connection between the algebra ${\cal A}_q$ and the reflection equation algebra introduced in \cite{Cher84,Skly88}. Below we exhibit a generating function for those elements that commute with the currents (\ref{c1}),(\ref{c2}).
\begin{prop}\label{propdelta} The element 
\beqa
\Delta(u)&=& -(q-q^{-1})(q^2+q^{-2})\Big(\cW_+(u)\cW_+(uq) + \cW_-(u)\cW_-(uq)\Big) - \frac{(q-q^{-1})}{\rho} \Big(\cG_+(u)\cG_-(uq) + \cG_-(u)\cG_+(uq)\Big) \nonumber\\
&&\quad +(q-q^{-1})(u^2q^2+u^{-2}q^{-2})\Big(\cW_+(u)\cW_-(uq) + \cW_-(u)\cW_+(uq)\Big)\nonumber\\
&& \quad - \ \cG_+(u) - \cG_+(uq) - \cG_-(u) - \cG_-(uq)  \  \nonumber
 \label{deltau}
\eeqa
satisfies \ $[\Delta(u),\cW_\pm(v)]=[\Delta(u),\cG_\pm(v)]=0$.
\end{prop}
\begin{proof} By \cite[Theorem 1]{BSh1}, the defining relations (\ref{ec1})-(\ref{ec16}) with $\rho=k_+k_-(q+q^{-1})^2$ can be written in the `compact' form of a reflection equation \cite{Skly88}:
\begin{align} R(u/v)\ (K(u)\otimes I\!\!I)\ R(uv)\ (I\!\!I \otimes K(v))\
= \ (I\!\!I \otimes K(v))\ R(uv)\ (K(u)\otimes I\!\!I)\ R(u/v),
%\qquad \forall u,v\ .
\label{RE} \end{align}
where $ I\!\!I$ denotes the $2\times 2$ identity matrix, 
\begin{align}
R(u) =\left(
\begin{array}{cccc} 
 uq -  u^{-1}q^{-1}    & 0 & 0 & 0 \\
0  &  u -  u^{-1} & q-q^{-1} & 0 \\
0  &  q-q^{-1} & u -  u^{-1} &  0 \\
0 & 0 & 0 & uq -  u^{-1}q^{-1}
\end{array} \right) \ \label{R}
\end{align}
and
\beqa
K(u) =\left(
\begin{array}{cc}
  uq \cW_+(u) - u^{-1}q^{-1} \cW_-(u) & \frac{1}{k_-(q+q^{-1})}\cG_+(u) + \frac{k_+(q+q^{-1})}{(q-q^{-1})}  \\
  \frac{1}{k_+(q+q^{-1})}\cG_-(u) + \frac{k_-(q+q^{-1})}{(q-q^{-1})}\   &  uq \cW_-(u) - u^{-1}q^{-1} \cW_+(u)  \\
\end{array} \right),\label{Kmat}
\eeqa
where $k_\pm$ are nonzero scalars.\vspace{1mm} 

By \cite[Proposition 5]{Skly88}, the so-called quantum determinant:
\beqa
\Gamma(u)=tr\big(P^{-}_{12}(K(u)\otimes I\!\!I)\ R(u^2q) (I\!\!I \otimes K(uq))\big),\ \label{gamma}
\eeqa
where $P^-_{12}=(1-P)/2$ with $P=R(1)/(q-q^{-1})$, is commuting with the four entries of $K-$matrix (\ref{Kmat}): 
\beqa
\big[\Gamma(u),(K(u))_{ij}\big]=0 \quad \mbox{for any} \quad i,j. \label{qdep}
\eeqa
Inserting  (\ref{Kmat}) into (\ref{gamma}), by straightforward calculations one identifies:
\beqa
\Gamma(u)= \frac{(u^2q^2-u^{-2}q^{-2})}{2(q-q^{-1})}\left(\Delta(u) - \frac{2\rho}{(q-q^{-1})}\right),\nonumber
\eeqa
with $\Delta(u)$ as given above. From (\ref{qdep}), it follows $[\Delta(u),\cW_\pm(v)]=[\Delta(u),\cG_\pm(v)]=0$.
\end{proof}
%
%\begin{rem}\label{remark} 
%$\Omega(\Delta(u))=\Delta(u)$. 
%\end{rem}
%
\vspace{1mm}

From Proposition \ref{propdelta}, $\Delta(u)$  provides a generating function for elements that commute with all generators $\{{\cW}_{-k},{\cW}_{k+1},{\cG}_{k+1},{\tilde{\cG}}_{k+1}|k\in {\mathbb Z}_+\}$ of  ${\cal A}_q$. Explicit expressions are now derived.
\begin{lem}\label{lemdetta}
 Let $i,j,k,l\in {\mathbb Z}_+$ and denote $\bar{k}=1$ (resp. $0$) for $k$ even (resp. odd)  and $\null[\frac{k}{2}]=\null \frac{k}{2}$ (resp. $\frac{k-1}{2}$)  for   $k$ even (resp. odd). Define
\beqa
\qquad \Delta_{k+1} &=&{\mathbb G}_{k+1}+ \ \sum_{l=0}^{\null[\frac{k}{2}]-1}d^{(k)}_{l}{\mathbb G}_{2(l+1)-\bar{k}} +\sum_{l=0}^{\null[\frac{k}{2}]}\sum_{i+j=2l+1-\bar{k}} f^{(k)}_{ij} {\mathbb W}_{ij}
+\sum_{l=0}^{\null[\frac{k}{2}]-\bar{k}}\sum_{i+j=2l+\bar{k}}  e^{(k)}_{ij}{\mathbb F}_{ij}\label{deltak}
\eeqa
where
\beqa
{\mathbb G}_{i+1}&=& \cG_{i+1} + \tilde{\cG}_{i+1}, \nonumber\\
 {\mathbb W}_{ij}&=&(q-q^{-1})(\cW_{-i}\cW_{j+1} +\cW_{i+1}\cW_{-j} ),\nonumber\\
{\mathbb F}_{ij}&=&(q-q^{-1})\Big((q^2+q^{-2})(\cW_{-i}\cW_{-j}+\cW_{i+1}\cW_{j+1})+\frac{1}{\rho} (\cG_{j+1}\tilde{\cG}_{i+1} + \tilde{\cG}_{j+1}\cG_{i+1})\Big)\nonumber
\eeqa
and
\ben
e^{(k)}_{ij}&=&-c_{k+1}^{-1}\sum_{m=0}^{\null[\frac{k}{2}]-\frac{i+j+\bar k}{2}}w_{\null[\frac{k}{2}]-\frac{i+j+\bar k}{2}-m}^iw_m^jq^{-2j-4m-2},\nonumber\\
f^{(k)}_{ij}&=&c_{k+1}^{-1}\sum_{m=0}^{\null[\frac{k}{2}]-\frac{i+j+\bar k-1}{2}}w_{\null[\frac{k}{2}]-\frac{i+j+\bar k-1}{2}-m}^i \Big(w_m^j+w_{m-1}^j\Big)q^{-2j-4m} ,\nonumber\\
d^{(k)}_l &=&-c_{k+1}^{-1} w_{\null[\frac{k}{2}]-l}^{2l+1-\bar{k}}(1+q^{-2k-2})\nonumber 
\een
with 
\beqa
c_{k+1}=-\frac{(q+q^{-1})^{k+1}(q^{k+1}+q^{-k-1})}{q^{2k+2}}, \quad w_{m}^i=(-1)^m\frac{(m+i)!}{m!i!}(q+q^{-1})^{i+1}q^{-i-2m-1}.\nonumber
\eeqa
 For any $k$ and all  ${\cal X}\in {\cal A}_q$ we have  $[\Delta_{k+1},{\cal X}]=0$. 
\end{lem}
\begin{proof}
Insert (\ref{c1}), (\ref{c2}) in $\Delta(u)$ as defined in Proposition \ref{propdelta}. Observe that $U$ admits a power series expansion about $u=\infty$ or $u=0$. About $u=\infty$,  $\Delta(u)$ admits the power series expansion in $u^{-1}$ of the form $\sum_{k=0}^{\infty} u^{-2k-2}c_{k+1}\Delta_{k+1}$. The coefficients  $e^{(k)}_{ij} ,f^{(k)}_{ij}, d^{(k)}_l$ are determined as follows. Consider the power series expansion of $U(u)$ about $u=\infty$: 
\ben
U(u)^{-1-i}=\sum_{m=0}^\infty w_{m}^i u^{-2i-4k-2}. \nonumber
\een
From the definition of $\Delta(u)$ in Proposition \ref{propdelta}, the expansion about $u=\infty$ yields to:
\ben
\Delta(u)&=& \sum_{k=0}^{\infty}u^{-2k-2} \Big( \sum_{i+2j=k} w_{j}^i(1+q^{-2k-2}) {\mathbb G}_{i+1}+ \sum_{i+j+2(n+m)=k} (w_{n}^iw_{m}^j+w_{n}^iw_{m-1}^j)q^{-2j-4m+1} {\mathbb W}_{ij}\nonumber\\
&&\quad+ \sum_{i+j+2(n+m)=k-1} w_{n}^iw_{m}^jq^{-2j-4m-1} {\mathbb F}_{ij}\Big)\nonumber\\
\nonumber&=& \sum_{k=0}^{\infty}u^{-2k-2}c_{k+1} \Big( {\mathbb G}_{k+1}+\sum_{l=0}^{\null[\frac{k}{2}]-1} \Big( \frac{ w_{\null[\frac{k}{2}]-l}^{2l+1-\bar{k}}(1+q^{-2k-2})}{c_{k+1}}  \Big){\mathbb G}_{2(l+1)-\bar{k}}\\
&&\quad +\sum_{l=0}^{\null[\frac{k}{2}]} \sum_{i+j=2l+1-\bar k} \Big(\sum_{m=0}^{\null[\frac{k}{2}]-\frac{i+j+\bar k-1}{2}}\Big(w_{\null[\frac{k}{2}]-\frac{i+j+\bar k-1}{2}-m}^iw_m^j+w_{\null[\frac{k}{2}]-\frac{i+j+\bar k-1}{2}-m}^iw_{m-1}^j\Big)\frac{q^{-2j-4m}}{c_{k+1}}\Big){\mathbb W}_{ij}\nonumber\\
&&\quad+\sum_{l=0}^{\null[\frac{k}{2}]-\bar k} \sum_{i+j=2l+\bar k} \Big( \sum_{m=0}^{\null[\frac{k}{2}]-\frac{i+j+\bar k}{2}}w_{\null[\frac{k}{2}]-\frac{i+j+\bar k}{2}-m}^iw_m^j\frac{q^{-2j-4m-2}}{c_{k+1}}\Big){\mathbb F}_{ij}\Big)\nonumber\\
&=&\sum_{k=0}^{\infty} u^{-2k-2}c_{k+1}\Delta_{k+1}\nonumber.
\een
\\
Then,  it follows from Proposition \ref{propdelta} and  (\ref{c1})-(\ref{c2}), that $\Delta_{k+1}$ is central in ${\cal A}_q$.
Note that the same central elements $\Delta_{k+1}$ can be also derived by considering the power series expansion of $\Delta(u)$ about $u=0$.
\end{proof}
\begin{example}
The leading terms in the power series expansion of $\Delta(u)$ are given by:
\beqa
\qquad \qquad \Delta_1&=& \cG_{1} + \tilde{{\cG}}_{1} -(q-q^ {-1})\big(\cW_0\cW_1+\cW_1\cW_0\big) ,\label{delta1}\\ 
\Delta_2&=&  \cG_{2} + \tilde{{\cG}}_{2} - \frac{(q^2-q^{-2})}{(q^2+q^{-2})}( q^{-1}\cW_0\cW_2 + q\cW_2\cW_0 + q^{-1}\cW_1\cW_{-1} + q\cW_{-1}\cW_{1} ) \,   \label{delta2}\\
&&\qquad\qquad + \frac{(q-q^{-1})}{(q^2+q^{-2})}\,\Big((q^2+q^{-2})(\cW_0^2+\cW_1^2)+\frac{\tilde{\cG_1}\cG_1 + \cG_1\tilde{\cG}_1}{\rho}\Big),\nonumber
\eeqa
\beqa
\Delta_3&=&  \cG_{3} + \tilde{{\cG}}_{3}  - \frac{(q-q^{-1})}{(q^2+q^{-2}-1)}( q^{-2}\cW_0\cW_3 + q^2\cW_{3}\cW_0 + q^{-2}\cW_1\cW_{-2} + q^2\cW_{-2}\cW_{1}) \label{delta3}\\
&&\qquad\qquad - \frac{(q-q^{-1})}{(q^2+q^{-2}-1)} (\cW_{2}\cW_{-1} + \cW_{-1}\cW_{2} )   \nonumber\\
&&\qquad\qquad + \frac{(q-q^{-1})}{(q^2+q^{-2}-1)}\,\Big((q^2+q^{-2})(\cW_0\cW_{-1}+\cW_{1}\cW_2)+\frac{\tilde{\cG_2}\cG_1 + \cG_2\tilde{\cG_1}}{\rho}\Big) \nonumber\\ 
&&\qquad\qquad - \frac{1}{(q+q^{-1})^2}\ \Big(\cG_{1} + \tilde{{\cG}}_{1} -(q-q^ {-1})\big(\cW_0\cW_1+\cW_1\cW_0\big)\Big).\nonumber 
\eeqa
\end{example}
\begin{rem}\label{rem3} For all  $k\in {\mathbb Z}_+$, 
$\Omega({\Delta}_{k+1})=\Delta_{k+1}$\ . 
\end{rem}

\vspace{1mm}

\section{A class of quotients of ${\cal A}_q$ and the polynomial formulae}
In this section, we investigate a class of quotients of ${\cal A}_q$. A quotient in this class is denoted $\tilde{\cal A}^{\{\delta\}}_q$ (see Definition \ref{tildeAq}). In $\tilde{\cal A}^{\{\delta\}}_q$,  we are going to show that any generator $\{{\cW}_{-k},{\cW}_{k+1},{\cG}_{k+1},{\tilde{\cG}}_{k+1}|k\in{\mathbb Z}_+\}$ is a certain polynomial of the two basic generators ${\cW}_{0},{\cW}_{1}$ only. The polynomial formulae are obtained explicitly (see Proposition \ref{prop1}) and are unique modulo the  relations (\ref{ec1})-(\ref{ec16}). \vspace{1mm}
\begin{defn}\label{tildeAq}  
Let $\delta_{k+1}$, $k\in{\mathbb Z}_+$ be fixed scalars in ${\mathbb K}$.
The algebra ${\tilde{\cal A}}^{\{\delta\}}_q$ is defined as the quotient of the algebra ${\cal A}_q$ by the ideal generated by the relations $\{\Delta_{k+1}=2\delta_{k+1}|\forall k\in {\mathbb Z}_+\}$.
\end{defn}

We start the analysis by considering the defining relations (\ref{ec1})-(\ref{ec16}). Using the power series expansion (\ref{c1}), (\ref{c2}), the defining relations of  ${\cal A}_q$ can be written explicitly in terms of the generators $\{{\cW}_{-k},{\cW}_{k+1},{\cG}_{k+1},{\tilde{\cG}}_{k+1}|k\in{\mathbb Z}_+\}$, see \cite[Definition 3.1]{BSh1}. For the analysis below, it will be  sufficient to stare at four of these relations. For all $k\in {\mathbb Z}_+$, they read:
\begin{align}
\big[{\cW}_0,{\cW}_{k+1}\big]=\big[{\cW}_{-k},{\cW}_{1}\big]=\frac{1}{(q+q^{-1})}\big({\tilde{\cG}_{k+1} } - {{\cG}_{k+1}}\big)\ ,\label{qo1}\\
\big[{\cW}_0,{\cG}_{k+1}\big]_q=
%\big[{\tilde{\cG}}_{k+1},{\cW}_{0}\big]_q=
\rho{\cW}_{-k-1}-\rho{\cW}_{k+1}\ ,\label{qo2}\\
\big[{\cG}_{k+1},{\cW}_{1}\big]_q=
%\big[{\cW}_{1},{\tilde{\cG}}_{k+1}\big]_q=
\rho{\cW}_{k+2}-\rho{\cW}_{-k}\ ,\label{qo3}\\
\big[{\cW}_{0},{\cW}_{-k}\big]=0\ ,\quad 
\big[{\cW}_{1},{\cW}_{k+1}\big]=0\ .\label{qo4}
\end{align}

We first describe the straightforward consequences of the conditions $\Delta_{k+1}=2\delta_{k+1}$ for $k=0,1$ in addition to the subset of relations (\ref{qo1})-(\ref{qo4}). Imposing $\Delta_1=2\delta_{1}$ in (\ref{delta1})  and  considering eqs. (\ref{qo1})  for $k=0$, the two equations read, respectively:
\beqa
\cG_{1} + \tilde{{\cG}}_{1} &-&(q-q^ {-1})\big(\cW_0\cW_1+\cW_1\cW_0\big) -2\delta_1=0,\nonumber\\
\cG_{1} - \tilde{{\cG}}_{1} &-& (q+q^ {-1})\big(\cW_1\cW_0-\cW_0\cW_1\big)=0.\nonumber
\eeqa
It implies
\beqa
 \cG_{1} = [\cW_1,\cW_0]_q +\delta_1\qquad  \mbox{and}\qquad \tilde{\cG}_1=[\cW_0,\cW_1]_q + \delta_1.\label{recG1}
\eeqa
%
%It is straightforward to check that
%$\big[{\cW}_0,{\cG}_{1}\big]_q=\big[\tilde{\cG}_{1},{\cW}_{0}\big]_q$ and $\big[{\cG}_{1},{\cW}_{1}\big]_q= \big[{\cW}_1,%\tilde{\cG}_{1}\big]_q$,
%in agreement with the first equalities in the l.h.s of (\ref{qo2}) and (\ref{qo3}). 
Then, the r.h.s of  (\ref{qo2}), (\ref{qo3}) determine uniquely the next two generators in terms of $\cW_0,\cW_1,\cG_1$:
\beqa
{\cW}_{-1}&=&\frac{1}{\rho}\big[{\cW}_0,{\cG}_{1}\big]_q+{\cW}_{1}\ ,\qquad  {\cW}_{2}=\frac{1}{\rho}\big[{\cG}_{1},{\cW}_{1}\big]_q+{\cW}_{0}\ .\label{recW1}
\eeqa
Note that the first equality in  (\ref{qo1}) is satisfied for $k=1$. Indeed,  using (\ref{recG1}), (\ref{recW1}) one finds $\big[{\cW}_{0},{\cW}_{2}\big]=\big[{\cW}_{-1},{\cW}_{1}\big]$.
Consider now the first and second  eqs. of  (\ref{qo4}) for $k=1$. Respectively, they read\footnote{Note that inserting  (\ref{recG1}) into (\ref{WW}), one finds that $\cW_0,\cW_1$ satisfy the `$q-$Dolan-Grady relations':
\beqa
[\cW_0,[\cW_0,[\cW_0,\cW_1]_q]_{q^{-1}}]=\rho[\cW_0,\cW_1],\qquad
[\cW_1,[\cW_1,[\cW_1,\cW_0]_q]_{q^{-1}}]=\rho[\cW_1,\cW_0].
\label{qDGW} 
\eeqa
}
\beqa
\big[{\cW}_{0},{\cW}_{-1}\big]=0 ,\qquad \big[{\cW}_{1},{\cW}_{2}\big]=0. \label{WW}
\eeqa

According to eqs. (\ref{recG1}) (resp. (\ref{recW1})), we conclude that the generators ${\cG}_{1},\tilde{\cG}_{1}$ (resp. ${\cW}_{-1},{\cW}_{2}$) are uniquely determined as polynomials of total degree $2$ (resp. of degree less that $3$) in ${\cW}_{0},{\cW}_{1}$.

\vspace{1mm}

Next, we impose $\Delta_2=2\delta_2$.  Eqs.  (\ref{delta2}) and eqs. (\ref{qo1})  for $k=1$ become, respectively:
\beqa
 \cG_{2} + \tilde{{\cG}}_{2} &-& \frac{(q^2-q^{-2})}{(q^2+q^{-2})}\left( q^{-1}\cW_0\cW_2 + q\cW_2\cW_0 + q^{-1}\cW_1\cW_{-1} + q\cW_{-1}\cW_{1}\right) \,   \nonumber\\
\qquad\qquad &&+ \frac{(q-q^{-1})}{(q^2+q^{-2})}\,\Big((q^2+q^{-2})(\cW_0^2+\cW_1^2)+\frac{\tilde{\cG_1}\cG_1 + \cG_1\tilde{\cG}_1}{\rho}\Big) - 2\delta_2=0,\nonumber\\
\cG_{2} - \tilde{{\cG}}_{2} &-&(q+q^ {-1})\big(\cW_2\cW_0-\cW_0\cW_2\big)=0.\nonumber
\eeqa
It implies:
\beqa
\cG_2&=& \frac{(q^2-q^{-2})}{2(q^2+q^{-2})}\left( q^{-1}\cW_0\cW_2 + q\cW_2\cW_0 + q^{-1}\cW_1\cW_{-1} + q\cW_{-1}\cW_{1}\right)  + \frac{(q+q^{-1})}{2}\big[\cW_{2},\cW_{0}\big]    \label{recG2}\\
&& - \frac{(q-q^{-1})}{2(q^2+q^{-2})}\,\Big((q^2+q^{-2})(\cW_0^2+\cW_1^2)+\frac{\tilde{\cG_1}\cG_1 + \cG_1\tilde{\cG}_1}{\rho}\Big) + \delta_2, \nonumber
%&& + \frac{(q+q^{-1})}{2}\big[\cW_{2},\cW_{0}\big] + \delta_2\ ,\nonumber\\
%\cG_2&=&\frac{1}{(q^2+q^{-2})}\left(q\, [\cW_2,\cW_0]_{q^2}+q^{-1}\,  [\cW_1,\cW_{-1}]_{q^2}-(q-q^{-1})\,\Big(q^{-2} %\,(\cW_1)^2+q^2 \,(\cW_0)^2+\frac{\cG_1\,\tilde{\cG}_1}{\rho}\Big)\right)\ ,\nonumber\\
%\tilde{\cG}_2&=&\frac{1}{(q^2+q^{-2})}\left(q\, [\cW_{-1},\cW_1]_{q^2}+q^{-1}\,  [\cW_0,\cW_{2}]_{q^2}-(q-q^{-1})\,%\Big(q^{-2} \,(\cW_0)^2+q^2 \,(\cW_1)^2+\frac{\tilde{\cG}_1\,\cG_1}{\rho}\Big)\right)\ \nonumber
\eeqa
where the relation $\big[{\cW}_{0},{\cW}_{2}\big]=\big[{\cW}_{-1},{\cW}_{1}\big]$ has been used. Also, one finds:
\beqa
 \tilde{\cG}_2=\Omega(\cG_2).\nonumber
\eeqa
Then, eqs. (\ref{qo3}), (\ref{qo4})  determine uniquely the next two generators as polynomials of the  `lower' ones:
\beqa
{\cW}_{-2}&=&\frac{1}{\rho}\big[{\cW}_0,{\cG}_{2}\big]_q+{\cW}_{2}\ ,\qquad  {\cW}_{3}=\frac{1}{\rho}\big[{\cG}_{2},{\cW}_{1}\big]_q+{\cW}_{-1}\ .\label{recW2}
\eeqa
%
%Note that $\big[{\cW}_0,{\cG}_{2}\big]_q=\big[\tilde{\cG}_{2},{\cW}_{0}\big]_q$ and $\big[{\cG}_{2},{\cW}_{1}\big]_q= %\big[{\cW}_1,\tilde{\cG}_{2}\big]_q$, in agreement with the l.h.s of (\ref{qo2}) and (\ref{qo3}) for $k=1$. 
In addition, the first and second  eqs. of  (\ref{qo4}) for $k=2$ read:
\beqa
\big[{\cW}_{0},{\cW}_{-2}\big]=0 ,\qquad \big[{\cW}_{1},{\cW}_{3}\big]=0. \label{WW2}
\eeqa
Combining previous results, we conclude that, in $\tilde{\cal A}_q^{\{\delta\}}$, the generators ${\cG}_{2},\tilde{\cG}_{2}$ (resp. ${\cW}_{-2},{\cW}_{3}$) are uniquely determined as polynomials of total degree less than $4$ (resp. of degree less that $5$) in ${\cW}_{0},{\cW}_{1}$. Above analysis is generalized in a straightforward manner for $k\geq 2$:
\begin{prop}\label{prop1}   Let $i,j,k,l\in {\mathbb Z}_+$ and denote $\bar{k}=1$ (resp. $0$) for $k$ even (resp. odd)  and $\null[\frac{k}{2}]=\null \frac{k}{2}$ (resp. $\frac{k-1}{2}$)  for   $k$ even (resp. odd). One has:
\beqa
\cG_{k+1} &=& -\sum_{l=0}^{\null[\frac{k}{2}]}\sum_{i+j=2l+1-\bar{k}} \frac{f^{(k)}_{ij}}{2}{\mathbb W}_{ij}  - \sum_{l=0}^{\null[\frac{k}{2}]-\bar{k}}\sum_{i+j=2l+\bar{k}}  \frac{e^{(k)}_{ij}}{2}{\mathbb F}_{ij} -  \sum_{l=0}^{\null[\frac{k}{2}]-1}\frac{d^{(k)}_{l}}{2}{\mathbb G}_{2(l+1)-\bar{k}} \label{recGk}\\
 \qquad \qquad  && + \frac{(q+q^{-1})}{2}\big[\cW_{k+1},\cW_0\big] + \delta_{k+1}
\nonumber
\eeqa
where ${\mathbb W}_{ij} $, ${\mathbb G}_i$, ${\mathbb F}_{ij}$, $f^{(k)}_{ij}$, $e^{(k)}_{ij}$ and $d^{(k)}_{l}$ are given in Lemma  \ref{lemdetta} and
\beqa
\qquad {\cW}_{-k-1}\!\!\!&=&\frac{1}{\rho}\big[{\cW}_0,{\cG}_{k+1}\big]_q+{\cW}_{k+1}.\ \label{recWk}
\eeqa
The expressions for the other generators are given by ${\cW}_{k+1}=\Omega(\cW_{-k})$, $\tilde{\cG}_{k+1}=\Omega(\cG_{k+1})$.
\end{prop}
\begin{proof} We use an inductive argument. Assume $\Delta_{l+1}=2\delta_{l+1}$ for $l\leq k$.
Then, (\ref{deltak}) with  (\ref{qo1}) determine uniquely - modulo the relations (\ref{ec1})-(\ref{ec16}) - the generators ${\cG}_{k+1}$  as polynomials of the generators of lower degree, given by (\ref{recGk}). From Remark  \ref{remomega}, one has $\tilde{\cG}_{k+1}=\Omega(\cG_{k+1})$.  Given ${\cW}_{-k},{\cW}_{k+1},{{\cG}}_{k+1}$ for $k$ fixed, the relations (\ref{qo2})-(\ref{qo3}) determine uniquely ${{\cW}}_{-k-1}$ as (\ref{recWk}) and ${{\cW}}_{k+2}=\Omega(\cW_{-k-1})$.
\end{proof}

Iterating the recursive formulae (\ref{recGk}), (\ref{recWk}) and using the automorphism $\Omega$, it follows that each generator of $\tilde{\cal A}_q^{\{\delta\}}$  is a polynomial in $\cW_0,\cW_1$: 
\begin{cor}\label{cor31} The algebra $\tilde{\cal A}_q^{\{\delta\}}$ is generated by $\cW_0, \cW_1$.
\end{cor}
We now discuss more precisely how the original generators ${\cW}_{-k},{\cW}_{k+1},{{\cG}}_{k+1},{\tilde{\cG}}_{k+1}$  are related
to $\cW_0,\cW_1$. Let $\mathrm{d}$ denote the total degree of a monomial in the elements ${\cW}_0,{\cW}_1$. According to   (\ref{recG1}), ${\cG}_1,{\tilde{\cG}}_{1}$ are polynomials in the generators ${\cW}_0,{\cW}_1$ of degree 
$\mathrm{d}[{\cG}_1]=\mathrm{d}[{\tilde{\cG}}_{1}]= 2$. Then, from (\ref{recW1}) one finds $\mathrm{d}[{\cW}_{-1}]=\mathrm{d}[{\cW}_{2}]= 3$. By induction, according to Proposition \ref{prop1} it follows immediately: 
\begin{cor}\label{col} The generators of $\tilde{\cal A}_q^{\{\delta\}}$ are polynomials in ${\cW}_0,{\cW}_1$ of degree:
\beqa
\qquad \qquad \mathrm{d}[{\cW}_{-k}]=\mathrm{d}[{{\cW}}_{k+1}]= 2k+1 \qquad \mbox{and} \qquad \mathrm{d}[{\cG}_{k+1}]=\mathrm{d}[{\tilde{\cG}}_{k+1}]= 2k+2 \ , \qquad k\in{\mathbb Z}_+\ . \label{order}
\eeqa
\end{cor}
\begin{example}\label{ex1} The first generators read:
\beqa
\quad && \ {\cG}_{1} = q{\cW_1}\cW_0-q^{-1}{\cW_0}\cW_1 + \delta_1 ,\label{defel}\\
&&{\cW}_{-1} = \frac{1}{\rho}\left( (q^2+q^{-2})\cW_0\cW_1\cW_0 -\cW_0^2\cW_1 - \cW_1 \cW_0^2\right) + \cW_1 + \frac{\delta_1(q-q^{-1})}{\rho}\cW_0,\nonumber\\
&&{\cG}_{2} = \frac{1}{\rho(q^2+q^{-2})} \Big( (q^{-3}+q^{-1}) \cW_0^2{\cW_1}^2 - (q^{3}+q){\cW_1}^2\cW_0^2 + (q^{-3}-q^{3})(\cW_0{\cW_1}^2\cW_0 + {\cW_1}\cW_0^2{\cW_1})  \nonumber\\
&&\qquad \qquad  - (q^{-5}+q^{-3} +2q^{-1}) \cW_0{\cW_1}\cW_0{\cW_1} + (q^{5}+q^{3} +2q){\cW_1}\cW_0{\cW_1}\cW_0 +  \rho(q-q^{-1})(\cW_0^2 + {\cW_1}^2
)\Big) \nonumber\\
&&\qquad \qquad +\ \frac{\delta_1(q-q^{-1})}{\rho}\big(q{\cW_1}\cW_0-q^{-1}{\cW_0}\cW_1 \big)  + \delta_2- \frac{\delta_1^2(q-q^{-1})}{\rho(q^2+q^{-2})} .\nonumber 
\eeqa
Expressions of ${\cW}_{k+1}$ and ${\tilde{\cG}}_{k+1}$ are obtained using the automorphism $\Omega$.
\end{example}

For simplicity, the explicit expressions for the generators ${\cal W}_{-2},{\cal W}_{3},{\cal G}_{3},\tilde{\cal G}_3$ are reported in Appendix \ref{ApA}.

\begin{rem}
In \cite[Example 1]{BB0}, under certain assumptions polynomial formulae for the generators of ${\cal A}_q$ were obtained as the unique solutions of a system of linear equations (see \cite[Appendix]{BB0}), through a detailed analysis of the defining relations (\ref{ec1})-(\ref{ec16}). For $k=0,1,2$, we have checked that the polynomial formulae conjectured in  \cite{BB0} 
 coincide with the polynomial formulae  (\ref{recGk}) that holds in $\tilde{\cal A}_q^{\{\delta\}}$.
\end{rem} 
 
\vspace{1mm}

\section{Linear bases for the algebras $\tilde{\cal A}_q^{\{\delta\}}$ and ${\cal O}_q$}
In this section, a linear basis for the algebra $\tilde{\cal A}_q^{\{\delta\}}$ is conjectured (Conjecture \ref{conj1}). We call it  the $WG-$basis. Then, using the fact that the algebra $\tilde{\cal A}_q^{\{\delta\}}$ is a homomorphic image of the $q-$Onsager algebra ${\cal O}_q$, it is conjectured that the $WG-$basis  gives a basis for the $q-$Onsager algebra. A comparison between the $WG-$basis and the `zig-zag' basis of the $q-$Onsager algebra given by Ito-Terwilliger \cite{IT10} gives a support for the conjecture. \vspace{1mm}

First, we exhibit a natural spanning set for  the algebra $\tilde{\cal A}_q^{\{\delta\}}$. From the  relations (\ref{ec1}) and (\ref{ec9}) observe that $\{\cW_{-k}|k\in{\mathbb Z}_+\}$ (resp. $\{\cW_{k+1}|k\in{\mathbb Z}_+\}$, $\{\cG_{k+1}|k\in{\mathbb Z}_+\}$) generate a commutative subalgebra of $\tilde{\cal A}_q^{\{\delta\}}$.   Let $\{\alpha_i,\beta_i,\gamma_i,k_i,l_i,p_i\}\in{\mathbb Z}_+$. Consider monomials of the form:
\beqa
 \omega'_{\alpha,\beta,\gamma}(\{ \cW_{-k},\cW_{k+1},\cG_{k+1}|k\in{\mathbb Z}_+ \}) = 
\cW^{\alpha_1}_{-k_1}...\cW^{\alpha_N}_{-k_N}    \cG^{\beta_1}_{p_1+1}...\cG^{\beta_P}_{p_P+1}    \cW^{\gamma_M}_{l_M+1}...\cW^{\gamma_1}_{l_1+1}.\label{mon}
\eeqa
According to Proposition \ref{prop1} and Collorary \ref{col}, the monomial (\ref{mon}) is a polynomial in $\cW_0,\cW_1$ of maximum degree: 
\beqa
|\lambda'| = \sum_{i=1}^N\alpha_i (2k_i+1)+ \sum_{i=1}^P\beta_i (2p_i+2)+ \sum_{i=1}^M\gamma_i(2l_i+1).\label{degWG}
\eeqa

For small values of $|\lambda'|$, using the relations  (\ref{ec1})-(\ref{ec16}) it is 
 straightforward to identify\footnote{Note that a more explicit form for the commutation relations between the generators $\{\cW_{-k},\cW_{k+1},\cG_{k+1}\}$ has been used for this task, see \cite[Definition 3.1 and Proposition 3.1]{BSh1}.} the set of linearly independent  monomials of the form (\ref{mon}).

\begin{example}\label{Ex3}  Let $d_{|\lambda'|}$ denote the number of irreducible monomials of maximum degree $|\lambda'|$.  The irreducible monomials in $\tilde{\cal A}_q^{\{\delta\}}$ of degree $|\lambda'|\leq 6$ are the set:\\

\begin{center}
\begin{tabular}{|c|c| c | c | c|}
\hline
$|\lambda'|$ &  $WG-$basis  & $d_{|\lambda'|}$ \\
\hline 
$0$ &  1 & 1 
\\
$1$ &  $\cW_0$,\ $\cW_1$ &   2 
\\
$2$ &  $\cW_0^2$, $\cW_1^2$, $\cW_0\cW_1$, $\cG_1$   & 4 
\\
$3$ &  $\cW_0^3$, $\cW_1^3$, $\cW_0^2\cW_1$, $\cW_0\cW_1^2$, $\cW_0\cG_1$, $\cG_1\cW_1$, $\cW_{-1}$, $\cW_2$    & 8 
 \\

$4$ & $\cW_0^4$, $\cW_1^4$, $\cW_0^3\cW_1$, $\cW_0^2\cW_1^2$, $\cW_0\cW_1^3$, $\cW_0^2\cG_1$, $\cG_1\cW_1^2$, $\cW_0\cW_{-1}$, $\cW_2\cW_1$,
  & 14  \\
&  $\cW_0\cG_1\cW_1$, $\cW_{-1}\cW_1$, $\cW_0\cW_2$, $\cG_1^2$, $\cG_2$   &   \\
& & \\

$5$ & $\cW_0^5$, $\cW_0^4\cW_1$, $\cW_0^3\cW_1^2$, $\cW_0^2\cW_1^3$, $\cW_0^2\cG_1\cW_1$,  $\cW_{0}\cW_1^4$, $\cW_0\cG_1\cW_{1}^2$, $\cW_0\cG_1^2$ , $\cW_0\cG_2$, $\cW_{-2}$, & 24 \\
&  $\cW_{-1}\cG_{1}$,  $\cW_0\cW_2\cW_{1}$, $\cW_0\cW_{-1}\cW_1$,  $\cW_3$, 
   $\cW_0^2\cW_2$, $\cW_2\cW_1^2$,  $\cW_{-1}\cW_{1}^2$, $\cW_0^2\cW_{-1}$, & \\
& $\cW_1^5$, $\cG_1\cW_2$, $\cW_0^3\cG_1$, $\cG_1^2\cW_1$, $\cG_1\cW_2$, $\cG_2\cW_1$  &  \\

& & \\

$6$ &     $\cW_0^6$, $\cW_0^5\cW_1$, $\cW_0^4\cW_1^2$,  $\cW_0^4\cG_1$ , $\cW_0^3\cW_1^3$,  $\cW_0^3\cG_1\cW_1$, 
$\cW_0^3\cW_{-1}$, $\cW_0^3\cW_2$, $\cW_0^2\cW_1^4$,   $\cW_0^2\cG_1\cW_1^2$,  &  40 \\
&  $\cW_0^2\cW_{-1}\cW_1$,    $\cW_0^2\cW_2\cW_1$,  $\cW_0^2\cG_1^2$, $\cW_0^2\cG_2$, $\cW_0\cW_1^5$, $\cW_0\cG_1\cW_1^3$, $\cW_0\cW_{-1}\cW_1^2$,  & \\
& $\cW_0\cW_2\cW_1^2$, $\cW_0\cG_1^2\cW_1$, $\cW_0\cG_2\cW_1$, $\cW_0\cW_{-2}$, $\cW_0\cW_3$, $\cW_0\cW_{-1}\cG_1$, $\cW_0\cG_1\cW_2$&   \\
&  $\cG_1\cW_1^4$, $\cW_{-1}\cW_1^3$, $\cW_2\cW_1^3$ , $\cG_1^2\cW_1^2$, $\cG_2\cW_1^2$,  $\cW_1^6$, $\cW_{-2}\cW_1$, $\cW_3\cW_1$, $\cW_{-1}\cG_1\cW_1$, $\cG_1\cW_2\cW_1$,   & \\
&  $\cG_1^3$, $\cG_1\cG_2$, $\cG_3$, $\cW_{-1}^2$, $\cW_2^2$, $\cW_{-1}\cW_2$ &  \\

\hline
\end{tabular}
        \end{center}
\end{example}
\vspace{3mm}

This gives a support for the following conjecture:
\begin{conj}\label{conj1} Let $\{\alpha_i,\beta_i,\gamma_i,k_i,l_i,p_i\}\in{\mathbb Z}_+$. Define the ordering
\beqa
 k_1<...<k_N;\quad  p_1<...<p_P; \quad \l_1<...<l_M. \nonumber
\eeqa
Vectors of the form
\beqa
\cW^{\alpha_1}_{-k_1}...\cW^{\alpha_N}_{-k_N}    \cG^{\beta_1}_{p_1+1}...\cG^{\beta_P}_{p_P+1}    \cW^{\gamma_M}_{l_M+1}...\cW^{\gamma_1}_{l_1+1} \label{basis1} 
\eeqa
induce a basis of type Poincar\'e-Birkhoff-Witt of $\tilde{\cal A}_q^{\{\delta\}}$. We refer to this basis as the $WG-$basis.
\end{conj}
\vspace{1mm}

Consider the subset of irreducible monomials of the form (\ref{basis1}) for $|\lambda'|$ fixed. In the sum (\ref{degWG}), observe that the elements $\{\cW_{-k_i}\}$ contribute by odd integers, the elements $\{\cG_{p_i}\}$ by even integers and  the elements $\{\cW_{l_i+1}\}$ by odd integers. So, $d_{|\lambda'|}$ coincides with the number of partition of the integer $|\lambda'|$ where there are two kinds of odd parts. In the literature, this number is sometimes referred as the `number of partitions into red integers and blue odd integers'. The corresponding generating function is known, see for instance e.g. \cite[page 88]{GI}. It follows:

\begin{thm}  The generating function of the number of irreducible monomials in $\tilde{\cal A}_q^{\{\delta\}}$ is:
\beqa
\sum_{|\lambda'|=0}^{\infty} d_{|\lambda'|} v^{|\lambda'|}=\prod_{m=1}^{\infty}\frac{1+v^{m}}{1-v^m}.\label{genWG}
\eeqa 
\end{thm}
\vspace{2mm}

We now introduce the $q-$Onsager algebra denoted ${\cal O}_q$.
\begin{defn}[\cite{Ter03,B1}]\label{qons} Let $\rho$ be a nonzero scalar in ${\mathbb K}$.
The $q-$Onsager algebra is the associative algebra with unit and standard generators $\textsf{A},\textsf{A}^*$ subject to the so-called $q-$Dolan-Grady relations relations
\beqa
[\textsf{A},[\textsf{A},[\textsf{A},\textsf{A}^*]_q]_{q^{-1}}]=\rho[\textsf{A},\textsf{A}^*]\
,\qquad
[\textsf{A}^*,[\textsf{A}^*,[\textsf{A}^*,\textsf{A}]_q]_{q^{-1}}]=\rho[\textsf{A}^*,\textsf{A}]\
\label{qDG} . \eeqa
\end{defn}

A linear basis for the $q-$Onsager algebra ${\cal O}_q$ has been constructed in  \cite{IT10}. It is the so-called `zig-zag' basis. Let $r$ denote a positive integer. Define $ \Lambda_r = \{\lambda=(\lambda_0, \lambda_1,\dots, \lambda_r) \in \mathbb{Z}^{r+1}|\lambda_0\ge 0, \lambda_i \ge 1 (1 \le i \le r)\}$.
%\\
% \Lambda &=&\bigcup\limits_{r \in  \mathbb{N}\cup \{0\}} {\Lambda_r}. 
 %\end{eqnarray*}
%
 If there exists an integer $i~ (0 \le i\le r)$ such that 
 \[\lambda_0 < \lambda_1 < \dots <\lambda_i \ge \lambda_{i+1}\ge \dots \ge \lambda_{r},\]
 then $\lambda=(\lambda_0, \lambda_1, \dots, \lambda_r)$ is said to be irreducible. The sum $|\lambda|=\lambda_0+ \lambda_1, \dots+ \lambda_r$ is called the length of $\lambda$.
 % \begin{thm}\cite{IT10}
% \label{basis}
The following set is the zig-zag basis of the $q-$Onsager algebra  ${\cal O}_q$ as a ${\mathbb K}-$vector space \cite[Theorem 2.1]{IT10} (see \cite[Problem 3.4]{IT04}):
\beqa
 \{\omega_\lambda({\textsf A},{\textsf A}^*)|\ \lambda\  \text{is irreducible}\} \quad \mbox{where}\quad 
 \omega_\lambda({\textsf A},{\textsf A}^*)=\left\{\begin{array}{c}
 {\textsf A}^{\lambda_0}{{\textsf A}^*}^{\lambda_1}\dots {\textsf A}^{\lambda_r}~~~~ \text{ if}~ r~ \text{is even}\\
 {\textsf A}^{\lambda_0}{{\textsf A}^*}^{\lambda_1}\dots {{\textsf A}^*}^{\lambda_r}~~~~\text{if}~ r~ \text{is odd}
 \end{array}.\right.\label{zigzag}
\eeqa
%
 %\end{thm}

%\vspace{1mm}

We now review some results of \cite[Section 4]{T} that will be useful in the discussion below. For $0\leq |\lambda|<\infty$, let ${\cal O}_{|\lambda|}$ denote the subspace of ${\cal O}_q$ spanned by the words of length at most $|\lambda|$ in the generators ${\textsf A},{\textsf A}^*$.  The sequence $\{{\cal O}_{|\lambda|}\}_{|\lambda|\in {\mathbb Z}_+}$ is a filtration of ${\cal O}_q$. For $n\in {\mathbb Z}_+ $, consider the quotient ${\mathbb K}-$vector space $\overline{\cal O}_{|\lambda|}={\cal O}_{|\lambda|}/{\cal O}_{|\lambda|-1}$ and introduce the formal direct sum $\overline{\cal O}=\sum_{|\lambda|\in {\mathbb Z}_+}\overline{\cal O}_{|\lambda|}$. The sequence $\{\overline{\cal O}_{|\lambda|}\}_{|\lambda|\in {\mathbb Z}_+}$ is a ${\mathbb Z}_+-$grading of the algebra $\overline{\cal O}$. The algebra $\overline{\cal O}$ is called the graded algebra associated with the filtration $\{{\cal O}_{|\lambda|}\}_{|\lambda|\in {\mathbb Z}_+}$. By construction, one has:
\beqa  
\mathrm{dim}({\cal O}_{|\lambda|})=\mathrm{dim}(\overline{\cal O}_{|\lambda|}).\label{eqdim}
\eeqa

Let $U_q^+$ be the positive part of $U_q(\widehat{sl_2})$. By \cite[Theorem 4.4]{T}, recall that the algebra $\overline{\cal O}$ is isomorphic  to $U_q^+$. It follows that the generating function of $\mathrm{dim}(\overline{\cal O}_{|\lambda|})$ is equal to the formal character of the Verma module for $\widehat{sl_2}$ \cite[eq. (40)]{IT04}. From (\ref{eqdim}), for the $q-$Onsager algebra it follows\footnote{See also \cite[Note 4.7]{T}.}:
\begin{thm} The generating function of the number of irreducible monomials in ${\cal O}_q$ is:
\beqa
\sum_{|\lambda|=0}^{\infty} \mathrm{dim}({\cal O}_{|\lambda|}) v^{|\lambda|}=\prod_{m=1}^{\infty}(1-v^{2m})^{-1}(1-v^{2m-1})^{-2}.\label{genzz}
\eeqa
\end{thm} 
\vspace{1mm}

\begin{example}\label{Ex4}  The irreducible monomials in ${\cal O}_q$ of length $|\lambda|\leq 6$ are the set:\\
\vspace{1mm}

\begin{center}
\begin{tabular}{|c |c| c | c | c|}
\hline
$|\lambda|$ &  Zig-zag basis  &  $\mathrm{dim}({\cal O}_{|\lambda|})$\\
\hline 
 $0$ &  1 & 1 
\\
 $1$ &  ${\textsf A}$, ${\textsf A}^*$ &   2 
\\
 $2$ &  ${\textsf A}^2$, ${{\textsf A}^*}^2$, ${\textsf A}{\textsf A}^*$, ${\textsf A}^*{\textsf A}$  & 4 
\\
 $3$ & ${\textsf A}^3$, ${{\textsf A}^*}^3$, ${\textsf A}^2{{\textsf A}}^*$, ${\textsf A}{{\textsf A}^*}^2$, ${{\textsf A}^*}^2{\textsf A}$, ${{\textsf A}^*}{\textsf A}^2$,   ${\textsf A}{{\textsf A}^*}{\textsf A}$, ${{\textsf A}^*}{\textsf A}{{\textsf A}^*}$ & 8 
 \\

 $4$ & ${\textsf A}^4$, ${\textsf A}^3{\textsf A}^*$, ${\textsf A}^2{\textsf A}^*{\textsf A}$, ${\textsf A}^*{\textsf A}^3$, ${\textsf A}^2{{\textsf A}^*}^2$, ${\textsf A}{{\textsf A}^*}^2{\textsf A}$, ${\textsf A}{{\textsf A}^*}{\textsf A}{{\textsf A}^*}$, & 14  \\
 &   + ${\textsf A}\leftrightarrow {\textsf A}^*$ &  \\

 & &   \\

 $5$ &  
 ${\textsf A}^5$, ${\textsf A}^4{\textsf A}^*$, ${\textsf A}^3{{\textsf A}^*}{\textsf A}$, ${\textsf A}^*{\textsf A}^4$,
  ${\textsf A}^3{{\textsf A}^*}^2$,  ${\textsf A}^2{{\textsf A}^*}{\textsf A}{\textsf A}^*$, ${\textsf A}^2{{\textsf A}^*}^2{\textsf A}$, ${\textsf A}{{\textsf A}^*}^2{\textsf A}^2$, 
 & 24 \\

 &  ${{\textsf A}^*}^2{\textsf A}^3$, ${\textsf A}{{\textsf A}^*}{\textsf A}{\textsf A}^*{\textsf A}$, ${\textsf A}{{\textsf A}^*}^3{\textsf A}$, ${{\textsf A}^*}{\textsf A}^2{{\textsf A}^*}^2$, &   \\
 
  &   + ${\textsf A}\leftrightarrow {\textsf A}^*$ &  \\
  
   & &   \\

 $6$ & ${\textsf A}^6$, ${\textsf A}^5{\textsf A}^*$, ${\textsf A}^4{\textsf A}^*{\textsf A}$, ${\textsf A}^*{\textsf A}^5$, ${\textsf A}^4{{\textsf A}^*}^2$,   ${\textsf A}^3{\textsf A}^*{\textsf A}{\textsf A}^*$, ${\textsf A}^3{{\textsf A}^*}^2{\textsf A}$,  ${\textsf A}^2{{\textsf A}^*}{\textsf A}{\textsf A}^*{\textsf A}$, 
  
 & 40 \\
&
 ${\textsf A}^2{{\textsf A}^*}^2{\textsf A}^2$,  ${\textsf A}{{\textsf A}^*}^2{\textsf A}^3$, ${{\textsf A}^*}^2{\textsf A}^4$, ${\textsf A}^*{\textsf A}^4{\textsf A}^*$, 
 ${\textsf A}^*{\textsf A}^3{\textsf A}^*{\textsf A}$, ${\textsf A}^3{{\textsf A}^*}^3$,  ${\textsf A}^2{{\textsf A}^*}^2{\textsf A}{\textsf A}^*$,  & \\

& ${\textsf A}^2{{\textsf A}^*}^3{\textsf A}$, ${\textsf A}{{\textsf A}^*}^3{\textsf A}^2$,  ${\textsf A}{{\textsf A}^*}^2{\textsf A}{\textsf A}^*{\textsf A}$, ${\textsf A}{{\textsf A}^*}^2{\textsf A}^2{\textsf A}^*$, ${\textsf A}{\textsf A}^*{\textsf A}{\textsf A}^*{\textsf A}{\textsf A}^*$,
 & \\
&  + ${\textsf A}\leftrightarrow {\textsf A}^*$ &\\
\hline
\end{tabular}
        \end{center}
\end{example}
\vspace{3mm}

There is a close relationship between the algebra $\tilde{\cal A}_q^{\{\delta\}}$ and  the $q-$Onsager algebra ${\cal O}_q$. Indeed, following Corollary \ref{cor31} the algebra $\tilde{\cal A}_q^{\{\delta\}}$ admits a presentation with two generators $\cW_0,\cW_1$ subject to the relations (\ref{ec1})-(\ref{ec16}) with (\ref{recGk}), (\ref{recWk}).  Among the corresponding infinite family of polynomial relations satisfied by $\cW_0,\cW_1$ inherited from (\ref{ec1})-(\ref{ec16}), the relations of lowest degree are given by (\ref{qDG}) with   ${\textsf A}\rightarrow\cW_0, {\textsf A}^*\rightarrow\cW_1$, see eqs. (\ref{qDGW}):
\begin{lem}
 The basic generators ${\cW}_0,{\cW}_1$  of $\tilde{\cal A}_q^{\{\delta\}}$ satisfy the defining relations of the $q$-Onsager algebra.
\end{lem}
\begin{proof} 
Insert the polynomial  expressions of ${\cW}_{-1}$ (resp. ${\cW}_{2}=\Omega({\cW}_{-1})$) given in (\ref{defel}) in the commutation relations (\ref{qo4}) for $k=0,l=1$.
\end{proof}

 It follows:
\begin{prop}\label{propmap}
The map $\Psi : {\cal O}_q \to \tilde{\cal A}_q^{\{\delta\}}$ such that
\beqa
\Psi({\textsf A}) = \cW_0, \quad \Psi( {\textsf A}^*)=\cW_1\label{map}
\eeqa
is a surjective homomorphism.
\end{prop}

We now compare the irreducible set of monomials in the $WG-$basis (\ref{basis1}) to the ones in the
zig-zag basis (\ref{zigzag}).  First, it is elementary to check that  both generating functions (\ref{genWG}) and (\ref{genzz}) coincide. Secondly, using the data in Example \ref{Ex3} and Example \ref{Ex4},
the transition matrix from one basis to another is considered up to $|\lambda|=|\lambda'|=4$, see Appendix \ref{ApB}. It is found that this matrix of size $29*29$ is invertible. 
\begin{conj}\label{conj2} The map $\Psi$ in Proposition \ref{propmap} is a bijection.
\end{conj}

\vspace{0.5cm}
\noindent{\bf Acknowledgements:}  
We thank P. Terwilliger for discussions, useful comments and suggestions. P.B. thanks K. Raschel for discussions. P.B. is supported by C.N.R.S. 
   S.B. thanks the LMPT for hospitality, where part of this work has been done.
S.B. is supported by a public grant as part of the Investissement d'avenir project, reference ANR-11-LABX-0056-LMH,
LabEx LMH.
\vspace{0.2cm}

\newpage

\begin{appendix}

\section{The generators ${\cal W}_{-2},{\cal W}_{3},{\cal G}_{3},\tilde{\cal G}_3$ }\label{ApA}
Recall that the generators ${\cal G}_{1}$, ${\cal W}_{-1}$, ${\cal G}_{2}$ are given explicitly in terms of $\cW_0,\cW_1$ in Example \ref{ex1}, and $\tilde{\cal G}_{1}=\Omega({\cal G}_{1})$, ${\cal W}_{-1}=\Omega({\cal W}_{2})$, $\tilde{\cal G}_{2}=\Omega({\cal G}_{2})$. In addition, one finds:  
\beqa
&&{\cW}_{-2} = w_1\cW_0^3{\cW_1}^2 + w_2{\cW_0}^2\cW_1^2\cW_0  + w_3\big(\cW_1{\cW_0}^2\cW_1\cW_0 + \cW_0^2\cW_1\cW_0\cW_1\big) + w_4\cW_0\cW_1\cW_0\cW_1\cW_0 \nonumber \\
&&\qquad\qquad+ w_5{\cW_1}^2{\cW_0}^3+w_6\cW_0{\cW_1}^2{\cW_0}^2+w_7\cW_1{\cW_0}^3\cW_1\nonumber\\
&&\qquad \qquad   + w_8\cW_0\cW_1^2+w_9 \cW_1^2\cW_0
 + w_{10}\big(\cW_0^2\cW_1 +\cW_1\cW_0^2\big)+\ w_{11} \cW_1\cW_0\cW_1 + w_{12}\cW_0\cW_1\cW_0 \nonumber\\
&&\qquad \qquad  +\ w_{13}\cW_0^3 + w_{14}\cW_0+ \ w_{15}\cW_1\nonumber
\eeqa
where
\beqa
&&w_1=\frac{1}{\rho^2} , \quad w_2= - \frac{[2]_{q}[8]_{q}}{\rho^2[4]_{q}^2}, \quad w_3=-\frac{[4]_{q}}{\rho^2[2]_{q}},\quad  w_4=\frac{1}{\rho^2}\left(\frac{[2]_q[3]_q[8]_{q}}{[4]_q^2} +1\right),\nonumber\\
&& w_5=\frac{1}{\rho^2[3]_q}, w_6=-\frac{[2]_q[8]_q}{\rho^2[3]_q[4]_q},Ê\quad w_7 =\frac{[2]^2_q}{\rho^2[3]_q[4]_q},\quad  w_8= -\ \frac{1}{\rho},\quad w_9=-\ \frac{1}{\rho[3]_q},\nonumber\\
&& w_{10}=- \frac{(q-q^{-1})}{\rho^2}\delta_1,\quad  w_{11}=\frac{1}{\rho^2[3]_q}\left(\frac{[2]_q[3]_q[8]_{q}}{[4]_q^2} +1\right),\quad w_{12}=\frac{(q-q^{-1})[4]_{q}}{\rho^2[2]_{q}}\delta_1,\quad \nonumber\\
&& w_{13}= \frac{1}{\rho}\frac{(q-q^{-1})^2[2]_{q}}{[4]_{q}},\quad  w_{14} =1-\frac{(q-q^{-1})^2[2]_{q}}{\rho^2[4]_{q}}\delta_1^2+\frac{(q-q^{-1})}{\rho}\delta_2,\quad w_{15}= \frac{(q-q^{-1})}{\rho}\delta_1\nonumber
\eeqa
\vspace{0.3cm}

and $\cW_3=\Omega(\cW_{-2})$.

\vspace{1cm}
\beqa
\cG_3&=& g_1\cW_0^3\cW_1^3 + g_2\cW_0^2\cW_1^2\cW_0\cW_1 + g_3\cW_0^2\cW_1^3\cW_0 + g_4\cW_0\cW_1^3\cW_0^2 \nonumber\\
&&  + g_5\cW_0\cW_1^2\cW_0\cW_1\cW_0  + g_6\cW_0\cW_1^2\cW_0^2\cW_1 + g_7\cW_0\cW_1\cW_0\cW_1\cW_0\cW_1  \nonumber\\
&& +   g_8\cW_1^3\cW_0^3 + g_9\cW_1^2\cW_0^2\cW_1\cW_0 + g_{10}\cW_1^2\cW_0^3\cW_1 + g_{11}\cW_1\cW_0^3\cW_1^2\nonumber\\
&& + g_{12}\cW_1\cW_0^2\cW_1\cW_0\cW_1   + g_{13}\cW_1\cW_0^2\cW_1^2\cW_0 + g_{14}\cW_1\cW_0\cW_1\cW_0\cW_1\cW_0  \nonumber\\
&& + g_{15}\cW_0^3\cW_1 +  g_{16}\cW_0^2\cW_1\cW_0 +  g_{17}\cW_1\cW_0^3  + g_{18}\cW_0^2\cW_1^2
+  g_{19}\cW_0\cW_1^2\cW_0 +  g_{20}\cW_0\cW_1\cW_0\cW_1   \nonumber\\
&& + g_{21}\cW_1^3\cW_0 +  g_{22}\cW_1^2\cW_0\cW_1 +  g_{23}\cW_0\cW_1^3 + g_{24}\cW_1^2\cW_0^2
+  g_{25}\cW_1\cW_0^2\cW_1 +  g_{26}\cW_1\cW_0\cW_1\cW_0  \nonumber\\ &&+ g_{27}\cW_0^2 + g_{28}\cW_1^2+ g_{29}\cW_0\cW_1 + g_{30}\cW_1\cW_0  + g_{31}I\!\!I \nonumber
 \eeqa
where
\beqa
&& g_1= -\frac{q^{-3}[2]_{q}}{\rho^2[6]_{q}}
,\quad
 g_2= \frac{2(q^{-7}+q^{-3}+q^{-1})[2]^2_{q}[3]_{q}}{\rho^2[4]_{q}[6]_{q}},\nonumber\\
&&g_3= \frac{q^{-1}[4]_{q}(q^{2}-q^{-2}-1)}{\rho^2[6]_{q}} ,\quad g_4=  \frac{(q^5+q^3-q-q^{-3}-q^{-5}-q^{-7})[2]_q^2}{\rho^2[4]_{q}[6]_q},\nonumber\\
&&g_5= -\frac{(q-q^{-1})[3]^2_{q}[4]_{q}}{\rho^2[6]_{q}}  ,\quad g_6= \frac{q^{-1}[2]_{q}[8]_{q}}{\rho^2[4]_{q}^2},\quad g_7= -\frac{(q^{-9}+q^{-7}+2q^{-5}+q^{-3}+3q^{-1})[2]_{q}^2[3]_{q}}{\rho^2 [4]_{q}[6]_{q}},\nonumber\\
&& g_8= \frac{q^3[2]_{q}}{\rho^2[6]_{q}},\quad  g_9= -\frac{2(q^7+q^3+q)[2]_{q}^2[3]_{q}}{\rho^2[4]_{q}[6]_{q}} ,\quad  \nonumber\\
&& g_{10}= -\frac{q[4]_{q}(q^{-2}-q^2-1)}{\rho^2[6]_{q}},\quad g_{11}=\frac{(q^7+q^5+q^3+q^{-1}-q^{-3}-q^{-5})[2]_{q}^2}{\rho^2[4]_{q}[6]_{q}},\nonumber\\
&& g_{12}=g_5,\quad g_{13}= -\frac{q[2]_q[8]_{q}}{\rho^2[4]^2_{q}} ,\quad g_{14}= \frac{(q^9+q^7+2q^5+q^3+3q)[2]_{q}^2[3]_{q}}{\rho^2[4]_{q}[6]_q},\nonumber
\eeqa
\beqa
&& g_{15}= -\frac{q^{-1}(2q^6+q^4+2q^{2}+1+4q^{-2}+2q^{-4}+2q^{-6})[2]_{q}}{\rho[3]_{q}^2[4]_{q}},\quad \nonumber\\
&& g_{16}= \frac{2q[6]_{q}}{\rho[3]_{q}[4]_{q}},\quad g_{17}= -\frac{q(q^6+2q^4+3q^{2}+q^{-2}-q^{-4}-q^{-6}-q^{-8})[2]_q^2}{\rho[3]_{q}[4]_{q}[6]_{q}} ,\nonumber\\
&& g_{18}=  \frac{q^{-2}(q-q^{-1})[2]^2_{q}}{\rho^2[4]_{q}}\delta_1 ,\quad g_{19}= -\frac{(q-q^{-1})^2[2]_{q}[3]_{q}}{\rho^2[4]_{q}}\delta_1,\nonumber\\
&& g_{20}= \frac{-q^{-3}(q-q^{-1})(q^2+ [3]_{q})[2]_{q}}{\rho^2[4]_{q}}\delta_1,\nonumber\\
&& g_{21}= \frac{q(2q^6+2q^4+4q^{2}+1+2q^{-2}+q^{-4}+2q^{-6})[2]_{q}}{\rho[3]_{q}^2[4]_{q}},\nonumber\\
&& g_{22}= -\frac{2q^{-1}[6]_{q}}{\rho[3]_{q}[4]_{q}} ,\quad g_{23}= -\frac{q^{-1}(q^8+q^6+q^{4}-q^{2}-3q^{-2}-2q^{-4}-q^{-6})[2]_{q}^2}{\rho[3]_{q}[4]_{q}[6]_{q}} ,\nonumber\\
&& g_{24}=-\frac{q^{2}(q-q^{-1})[2]^2_{q}}{\rho^2[4]_{q}}\delta_1 ,\quad g_{25}= g_{19} ,\nonumber\\
&& g_{26}= \frac{q^{3}(q-q^{-1})(q^{-2}+ [3]_{q})[2]_{q}}{\rho^2[4]_{q}}\delta_1,\nonumber\\
&& g_{27}= g_{28}= \frac{(q-q^{-1})^2[2]_{q}}{\rho[4]_{q}}\delta_1,\nonumber\\
&& g_{29}= -\frac{q^{-1}(q-q^{-1})}{\rho^2}\left(\rho\delta_2 -\frac{(q-q^{-1})[2]_{q}}{[4]_{q}} \delta_1^2\right) + \frac{q^3(q^4-q^{-8}-2q^{-6})[2]_{q}^2}{[3]_{q}[4]_{q}[6]_{q}},\nonumber\\
&&\quad g_{30}= \frac{q(q-q^{-1})}{\rho^2}\left(\rho\delta_2 -\frac{(q-q^{-1})[2]_{q}}{[4]_{q}} \delta_1^2\right) + \frac{q^{-3}(q^8-q^{-4}+2q^{6})[2]_q^2}{[3]_{q}[4]_{q}[6]_{q}},\nonumber\\
&& g_{31}= \delta_3 + \frac{1}{[2]_{q}^2}\delta_1+ \frac{(q-q^{-1})^2[2]_{q}^2[3]_{q}}{[4]_{q}[6]_{q}}\delta_1^3-\frac{(q-q^{-1})[2]_{q}[3]_{q}}{\rho[6]_{q}}\delta_1\delta_2  \nonumber
\eeqa
\vspace{0.5cm}

and $\tilde{\cG}_3=\Omega(\cG_{3})$.

\vspace{2mm}

\footnotesize

\begin{landscape}

\section{Transition coefficients between the zig-zag basis and $WG-$basis for $|\lambda|=|\lambda'|\leq 4$}\label{ApB}

In this Appendix, the non-vanishing entries of the transition matrix from the zig-zag basis with elements $1,{\textsf A},{\textsf A}^*,...$ to the $WG-$basis with elements $1,\cW_0,\cW_1,...$  are displayed for $|\lambda|=|\lambda'|\leq 4$ in the table below. For instance, Proposition \ref{propmap} together with the polynomial formulae (\ref{recG1}) give the combination:
$ q {\textsf A}{\textsf A}^*{\textsf A} - q^{-1}{\textsf A}^2{\textsf A}^* + \delta_1{\textsf A} \longrightarrow \cW_0\cG_1$.
\vspace{2mm}

\begin{center}
\begin{tabular}{|c | c | c | c | c | c | c| c | c | c | c | c| c | c | c | c | c | c | c | c | c | c | c | c | c | c| c | c | c  | c |c |c |c |c |}

 \hline 
 
 & $1$ & $A$  & $A^*$ & $A^2$ & ${A^*}^2$ & $AA^*$  & $A^*A$ & $A^3$ & ${A^*}^3$  & $A^2A^*$ &  $A{A^*}^2$ & $AA^*A$   & $A^*AA^*$ & $A^*A^2 $ & ${A^*}^2A$ & $A^4$  & ${A^*}^4$ & $A^3A^*$ & $A^2{A^*}^2$& $A{A^*}^3$ & \\
\hline 
$1$ & 1 &  & &  & &  & &  & &  & &  & & & &  & & & & &\\
 
 \hline 
$\cW_0$ &  & 1 & &  & &  & &  & &  & &  & & & &   & & & & &\\
 
 \hline 
$\cW_1$ &  &  & 1 &  & &  & &  & &  & &  & & & &  & & & & & \\
 
 \hline 
$\cW_0^2$ &  &  &  & 1 &  & & &  & &  & &  & & & &  & & & & & \\
 
 \hline 
$\cW_1^2$  &  &  &  &   &   1    &  & &  & &  & & & &  & &  & & & & & \\
 
 \hline 
$\cW_0\cW_1$ &  &  &  &   &    & 1 &   &  & &  & & & &  & & & &  & & & \\
 
 \hline 
$\cG_1$ & $\delta_1$ &  &  &   &    & $-q^{-1}$  & $q$ &  & & & &  & &  & & & & & & &\\
 
 \hline 
$\cW_0^3$ &  &  &  &   &   &   &   & 1  & &  & &  & & & &  & & & & &\\
 
 \hline 
$\cW_1^3$ &  &  & &  & &  & &  & 1  & & &  & & & & & & &  & &\\
 
 \hline 
$\cW_0^2\cW_1$  &  &  & &  & &  & &  & & 1  & &  & & & &   & & & & &\\
 
 \hline 
$\cW_0\cW_1^2$  &  &  & &  & & & &  & &  & 1 & & & & &  & & & & &\\
  
 \hline 
$\cW_0\cG_1$ &  & $\delta_1$ & &  & &  & &  & & $-q^{-1}$ & &  $q$ & &  & & & & &  & &\\
 
 \hline 
$\cG_1\cW_1$ &  &  & $\delta_1$ &  & &  & &  & &  & $-q^{-1}$ &  & $q$ &  & & & & &  & &\\
  \hline 

$\cW_{-1}$ & $a$ &  &  &  & &  & &  & &  $-\frac{1}{\rho}$ &  & $b$ &  & $-\frac{1}{\rho}$ & & & & & & &\\  
  
\hline
  
$\cW_{2}$ &  & $a$ &  &  & &  & &  &  &  & $-\frac{1}{\rho}$  &  & $b$ &  & $-\frac{1}{\rho}$ & & & & & &\\  
  
\hline  

$\cW_0^4$ &  &  &  &  & &  & &  & &  & &  &  &  & & $1$ & & & & &\\

\hline 

$\cW_1^4$ &  &  &  &  & &  & &  & &  & &  &  &  &  & & $1$ & & & & \\

\hline 

$\cW_0^3\cW_1$ &  &  &  &  & &  & &  & &  & &  & & &  & & & 1 & & & \\

\hline 

$\cW_0^2\cW_1^2$ &  &  &  &  & &  & &  & &  & &  & & &  &  & & & 1& &\\

\hline 

$\cW_0\cW_1^3$ &  &  &  &  & &  & &  & &  & &  & & &  & & &  & & 1& \\

\hline 

\end{tabular}
\end{center}

\vspace{9mm}

\begin{center}
\begin{tabular}{|c | c | c | c | c | c | c| c | c | c | c | c| c | c | c | c | c | c | c | c | c | c | c | c | c | c| c | c | c  | c |c |c |c |}

 \hline 
 
 & $1$ & $A^2$  & ${A^*}^2$ & $AA^*$ & $A^*A$ & $A^3A^*$  & $A^2{A^*}^2$ & $A{A^*}^3$ & & $A^2A^*A$  & ${A^*}^3A$ & $A^*A^3$ &  ${A^*}^2A A^*$  & $AA^*AA^*$ & $A^*A^2A^*$  & $A{A^*}^2A$ & $A^*AA^*A$ & ${A^*}^2A^2$  \\
\hline 

$\cW_0^2\cG_1$ &  & $\delta_1$ &  &  &  &  $-q^{-1}$ & &  & & $q$ &  & &  &  & & &   &  \\

\hline

$\cG_1\cW_1^2$  &  &  & $\delta_1$ & $-\frac{q\rho}{[3]_{q}}$ &  $\frac{q\rho}{[3]_{q}}$&  & &  $i$ & &  & $-\frac{q}{[3]_{q}}$ &  & $q$ & &  & & &  \\

 \hline 
$\cW_{0}\cW_{-1}$ &  &  $a$ &   & $\frac{\rho b}{[3]_{q}}$ & $\frac{1}{[3]_{q}}$ & - $\frac{b}{[3]_{q}}$ &   &  &   & $\frac{[2]_{q^2}-1}{\rho}$ &  & $-\frac{1}{\rho[3]_{q}}$ & & &  & & &  \\
 
 \hline 
$\cW_{2}\cW_{1}$ &  &  & $a$ & $\frac{1}{[3]_{q}}$ & $\frac{\rho b}{[3]_{q}}$ &  & & $-\frac{1}{\rho[3]_{q}}$  &   &   & $-\frac{b}{[3]_{q}}$ & & $\frac{[2]_{q^2}-1}{\rho}$ &  &  & & &  \\
 
 \hline 
$\cW_0\cG_1\cW_1$ &  &  &  & $\delta_1$ &  && $-q^{-1}$ &  & & &  & &  & $q$ &  & & &  \\

 \hline 
 $\cW_{-1}\cW_1$ &  &  & $1$ & $a$  &       && $-\frac{1}{\rho}$  &  & & &  & &  &  $b$ & $-\frac{1}{\rho}$ & &  & \\

 \hline 
$\cW_{0}\cW_{2}$ &  & $1$ &   &  $a$  &  & & $-\frac{1}{\rho}$   & & & & &  &  & $b$ & &$-\frac{1}{\rho}$ &  & \\

 \hline 
  $\cG_1^2$  & $\delta_1^2$ &  &  & $-2q^{-1}\delta_1$  &  $2q\delta_1$  &   &   &  & & &  & &  & $q^{-2}$ & $-1$& $-1$ & $q^2$ & \\
 
 \hline 
$\cG_2$ & $j$ & $e$ & $e$ & $-aq^{-1}$ & $aq$ & & $c$&  & &  &  & &  & $f$ & $h$ & $h$ & $g$ & $d$ \\

\hline

\end{tabular}
\end{center}

\vspace{3mm}
      
Here we define the scalars: $a=\frac{\delta_1(q-q^{-1})}{\rho}$, $b=\frac{[4]_{q}}{\rho[2]_{q}}$,
 $c=\frac{q^{-2}[2]^2_{q}}{\rho[4]_{q}}$, $d=-\frac{q^{2}[2]^2_{q}}{\rho[4]_{q}}$, $e=(q-q^{-1})\frac{[2]_{q}}{[4]_{q}}$,
 $f=-\frac{(q^{-5}+q^{-3}+2q^{-1})[2]_{q}}{\rho[4]_{q}}$, $g=\frac{(q^{5}+q^{3}+2q)[2]_{q}}{\rho[4]_{q}}$, $h=-(q-q^{-1})\frac{[2]_{q}[3]_{q}}{\rho[4]_{q}}$, $i=-\frac{q^{-2}[2]_{q}}{[3]_{q}}$,
$j=\delta_2-\delta_1^2(q-q^{-1})\frac{[2]_{q}}{\rho[4]_{q}}$.

\end{landscape}

\normalsize

\vspace{2cm}

\end{appendix}

\vspace{0.2cm}

\end{document}